\documentclass[11pt,a4paper,twoside,reqno]{amsart}
\usepackage[T1]{fontenc}
\usepackage{lmodern}
\usepackage[a4paper,inner=31mm,outer=34mm,top=30mm,bottom=32mm,headsep=8mm]{geometry}
\usepackage{amsmath,amssymb,mathtools,bm}
\usepackage[expansion=false]{microtype}
\usepackage[dvipsnames]{xcolor}
\usepackage[english]{babel}
\usepackage{hyperref}
\usepackage[nameinlink,capitalise,noabbrev]{cleveref}
\hypersetup{colorlinks=true,linkcolor=red,citecolor=ForestGreen,
urlcolor=blue,pdfborder={0 0 0},bookmarksopen=true,bookmarksnumbered=true,
pdftitle={Arithmetic Nonexistence Conditions for Tight Spherical $5$-Designs},
pdfsubject={Dual lattices, Gauss sums, and a symmetric-square obstruction}}
\newtheorem{theorem}{Theorem}[section]
\newtheorem{lemma}[theorem]{Lemma}
\theoremstyle{definition}
\newtheorem{definition}[theorem]{Definition}
\theoremstyle{remark}

\numberwithin{equation}{section}
\newcommand{\Z}{\mathbb Z}
\newcommand{\Q}{\mathbb Q}
\newcommand{\R}{\mathbb R}
\newcommand{\F}{\mathbb F}
\newcommand{\ii}{\mathrm i}
\title[Arithmetic Nonexistence Conditions for Tight Spherical $5$-Designs]
{Arithmetic Nonexistence Conditions for Tight Spherical $5$-Designs}

\author{Zili Xu}
\address{School of Mathematical Sciences,  Key Laboratory of MEA (Ministry of Education), Shanghai Key Laboratory of PMMP, and Nantong Institute for Applied Mathematics and Artificial Intelligence, East China Normal University, Shanghai 200241, China}
\email{zlxu@math.ecnu.edu.cn}

\begin{document}
\begin{abstract}
We prove two arithmetic nonexistence conditions for tight spherical $5$-designs in dimension $(2m+1)^2-2$. Let $m$ be a positive even integer satisfying $\nu_2(m(m+1))\leq5$, where $\nu_p(a)$ denotes the exponent of the prime $p$ in the positive integer $a$. We show that such a design does not exist if either (i) $m\equiv1\pmod3$ and $\nu_p(m(m+1))\leq1$ for every prime $p\geq7$ with $p\equiv3\pmod4$, or (ii) $m\equiv0\pmod3$, $\nu_3(m(m+1))=1$, and $\nu_p(m(m+1))\leq1$ for every prime $p\geq5$ with $p\not\equiv1\pmod{12}$. Together, these conditions recover the corresponding nonexistence results of Bannai, Munemasa, and Venkov and of Nebe and Venkov, and exclude parameters not covered by either earlier result, including $m=16,40,88,96,100$. The proofs use lattice theory to constrain the discriminant groups of the associated lattices, and derive contradictions through Gauss sums and determinant congruences.
\end{abstract}
\maketitle

\section{Introduction}\label{sec:introduction}
\subsection{Tight spherical \texorpdfstring{$5$}{5}-designs}
A finite nonempty set $D$ on the unit sphere $\mathbb{S}^{n-1}\subset\R^n$ is a \emph{spherical $t$-design} if
\[
\frac1{|D|}\sum_{\bm{x}\in D}f(\bm{x})
=\int_{\mathbb{S}^{n-1}}f(\bm{x})\,d\sigma(\bm{x})
\]
for every real polynomial $f$ of total degree at most $t$. Here $\sigma$ is surface measure normalized to have total mass $1$. Delsarte, Goethals, and Seidel proved that \cite{DGS}
\[
|D|\geq
\begin{cases}
\displaystyle\binom{n+k-1}{k}+\binom{n+k-2}{k-1},&t=2k,\\[6pt]
\displaystyle 2\binom{n+k-1}{k},&t=2k+1.
\end{cases}
\]
A spherical $t$-design is called \emph{tight} if equality holds in the corresponding bound \cite[Definition~5.13]{DGS}. In particular, a tight spherical $5$-design has $|D|=n(n+1)$, and such a design is antipodal, i.e., $\bm{x}\in D$ implies $-\bm{x}\in D$. Its antipodal pairs determine $n(n+1)/2$ equiangular lines, attaining the absolute bound in dimension $n$. For a survey of spherical designs, we refer to \cite{BBsurvey}. Related work includes constructions \cite{BannaiOrbits,HardinSloane}, asymptotic existence bounds \cite{BRV}, and energy questions \cite{CohnKumar}.

When $n\geq3$, tight spherical $5$-designs are known to exist in dimensions $3$, $7$, and $23$, with respectively $12$, $56$, and $552$ points, or $6$, $28$, and $276$ equiangular lines \cite[Example~8.3]{DGS}. These designs are unique up to orthogonal equivalence in their respective dimensions. See \cite[Example~5.16]{DGS} for dimension $3$ and the account in \cite[Section~1, p.~610]{BMV} for dimensions $7$ and $23$. The uniqueness in dimension $7$ also follows from \cite{BannaiSloane}, and the associated $276$-line system is studied through its regular two-graph in \cite{GoethalsSeidel}. For $n>3$, the necessary dimension condition is
\[
n=(2m+1)^2-2,\qquad m\in\Z_{>0}.
\]
See \cite{BMV} and the foundational tight-design restrictions in \cite{BDI,BDII,Bannai79}. The known designs in dimensions $7$ and $23$ correspond to $m=1$ and $m=2$.

The lattice method provides further nonexistence results within this necessary dimension family. Bannai, Munemasa, and Venkov proved nonexistence in dimension $47$, corresponding to $m=3$ \cite{BMV}. They also proved that no tight spherical $5$-design exists for
\[
m=2k,\qquad k\equiv2\pmod3,
\qquad k\text{ and }2k+1\text{ squarefree}
\]
\cite[Theorem~3.10]{BMV}. Here an integer is squarefree if no prime square divides it. In particular, $m=4$ is excluded, giving nonexistence in dimension $79$. Nebe and Venkov subsequently proved that no tight spherical $5$-design exists if $m$ is even, $8\nmid m$,  $m(m+1)$ is not divisible by the square of any odd prime, and $m\equiv0,1\pmod3$ \cite[Theorem~4.6]{NV}. In particular, it excludes $m=6$, as  recorded in \cite[Corollary~4.7]{NV}.

The common structure behind these arguments is an integral lattice generated by a suitably scaled half of the antipodal design, in the tradition of lattice and design methods \cite{NeumaierSeidel,Venkov,BachocVenkov}. The spherical moment identities restrict norms in its dual lattice. These restrictions in turn constrain the primary subgroups and squared norms of representatives in an associated finite quotient. The Milgram formula relates the resulting Gauss sum to the dimension modulo $8$.

There is also a graph-theoretic proof of the nonexistence of tight spherical $5$-designs in dimension $47$. If such a design exists, then there exists a strongly regular graph with parameters $(1127,486,165,243)$. Makhnev proved that a graph with these parameters cannot exist because the neighborhood of each vertex would be a strongly regular graph with parameters $(486,165,36,66)$, which he ruled out \cite{Makhnev}. The connection between equiangular lines and graph-theoretic structures is discussed in \cite{LemmensSeidel}.

\subsection{Our contributions}
We extend the arithmetic condition above in the two residue classes $m\equiv1\pmod3$ and $m\equiv0\pmod3$. For a prime $p$ and for a rational number $a=p^ku/v$ where $u,v$ are relatively prime integers with $p\nmid u$ and $p\nmid v$, we denote $\nu_p(a)=k$.

\begin{theorem}\label{thm:main}
Let $m$ be a positive even integer satisfying
\begin{gather*}
m\equiv1\pmod3,\qquad \nu_2(m(m+1))\leq5,\\
\nu_p(m(m+1))\leq1
\quad\text{for every prime }p\geq7\text{ with }p\equiv3\pmod4.
\end{gather*}
Then there is no tight spherical $5$-design in dimension $(2m+1)^2-2$.
\end{theorem}

\begin{theorem}\label{thm:second}
Let $m$ be a positive even integer satisfying
\begin{gather*}
m\equiv0\pmod3,\qquad \nu_3(m(m+1))=1,
\qquad \nu_2(m(m+1))\leq5,\\
\nu_p(m(m+1))\leq1
\quad\text{for every prime }p\geq5\text{ with }p\not\equiv1\pmod{12}.
\end{gather*}
Then there is no tight spherical $5$-design in dimension $(2m+1)^2-2$.
\end{theorem}

Theorem~\ref{thm:main} places no restriction on the valuations of $m(m+1)$ at primes $p\equiv1\pmod4$, and Theorem~\ref{thm:second} places no restriction on those at primes $p\equiv1\pmod{12}$. Both theorems also replace the condition $8\nmid m$ used in the even-parameter condition of Nebe and Venkov by the weaker bound $\nu_2(m(m+1))\leq5$. Hence, the new conditions cover parameter values that are not covered by either of the previous conditions in \cite{NV,BMV}. Among positive even integers $m\leq200$, the values satisfying the hypotheses of Theorem~\ref{thm:main} but not excluded by either previous condition are $ m\in\{16,40,88,100,112,124,136,160,184\}$. The corresponding list for Theorem~\ref{thm:second} is $m\in\{96,168\}$.

\subsection{Organization}
We use the notation $D,X,n,d,m,\Lambda,\Lambda_+,\Gamma$ of \cite{BMV}. Section~\ref{sec:preliminaries} records the elementary facts about dual-lattice cosets and Gauss sums needed in the proofs. Section~\ref{sec:lattice} records the design identities, constructs the even lattice $\Gamma=\Lambda_+/\sqrt2$ and proves the common arithmetic lemmas. Section~\ref{sec:first} proves Theorem~\ref{thm:main} by comparing real primary Gauss sums with a forced nonreal total phase. Section~\ref{sec:second} proves Theorem~\ref{thm:second} by investigating the integer $(\det\Lambda)/3$ modulo $3$.

\section{Preliminaries}\label{sec:preliminaries}

\subsection{Notations}
Let $\ii=\sqrt{-1}$. For a prime $p$, the $p$-adic valuation $\nu_p(a)$ of a nonzero rational number $a$ denotes the exponent of the prime $p$ in its reduced factorization. That is, if $a=p^ku/v$, where $u,v$ are relatively prime integers with $p\nmid u$ and $p\nmid v$, then $\nu_p(a)=k$. Matrices and vectors are written in bold italic type. We write $\bm I_t$ for the identity matrix of order $t$ and $\bm0$ for a zero vector or zero matrix of the indicated size. Inner products of real vectors are ordinary Euclidean inner products. For vectors $\bm v_1,\ldots,\bm v_r\in\R^d$, their Gram matrix is $(\langle\bm v_i,\bm v_j\rangle)_{i,j=1}^r\in\mathbb{R}^{r\times r}$. For matrices $\bm U=(U_{i,j})_{i,j=1}^n,\bm V=(V_{i,j})_{i,j=1}^n\in\R^{n\times n}$, define the entrywise inner product by
\[
\langle\bm U,\bm V\rangle_{\mathrm F}
:=\sum_{i=1}^n\sum_{j=1}^n U_{i,j}V_{i,j}.
\]
The Gram matrix of matrices $\bm U_1,\ldots,\bm U_r\in\R^{n\times n}$ is $(\langle\bm U_i,\bm U_j\rangle_{\mathrm F})_{i,j=1}^r\in\mathbb{R}^{r\times r}$.

For an integer $k$, write $k\Z=\{kr:r\in\Z\}$. If $k\ne0$, write $\frac1k\Z=\{\frac{r}{k}:r\in\Z\}$. For $a\in\Q$, the notation $a+\Z$ means the set $\{a+k:k\in\Z\}$. Hence, $t\in a+\Z$ means  that $t-a\in\Z$. For a set $A$, we use $A=S\sqcup T$ to represent a partition of $A$, where $S\cap T=\emptyset$ and $S\cup T=A$.

For an additive group $C$ and a subgroup $H$, the quotient $C/H$ consists of cosets $x+H=\{x+h:h\in H\}$, with $(x+H)+(y+H)=x+y+H$. For a finite abelian group $C$ and a prime $p$, define its Sylow $p$-subgroup by
\[
\operatorname{Syl}_p(C)=\{x\in C:p^a x=0\text{ for some integer }a\geq0\}.
\]
The order of $\operatorname{Syl}_p(C)$ is the largest power of $p$ dividing $|C|$. Each element of $C$ has a unique expression as a sum of elements from these Sylow subgroups.

\subsection{Lattice theory}

A   lattice in $\R^n$ has the form
\[
L=\left\{\sum_{j=1}^n c_j\bm{v}_j:c_j\in\Z\right\}
\]
with linearly independent vectors $\bm{v}_1,\ldots,\bm{v}_n\in\mathbb{R}^n$. Its rank is $n$. Its Gram matrix is $\bm{B}=(\langle\bm{v}_i,\bm{v}_j\rangle)_{i,j=1}^n\in\mathbb{R}^{n\times n}$, and its determinant is $\det L:=\det \bm{B}$. A lattice $L$ is \emph{integral} if $\langle\bm{v},\bm{w}\rangle\in\Z$ for every $\bm{v},\bm{w}\in L$, and \emph{even} if $\langle\bm{v},\bm{v}\rangle\in2\Z$ for every $\bm{v}\in L$. Its dual lattice is
\[
L^*=\{\bm{\alpha}\in\R^n:
\langle\bm{\alpha},\bm{v}\rangle\in\Z,
\forall\bm{v}\in L\}.
\]
If $L_0\subseteq L$ is a sublattice of $L$, then
\begin{equation}\label{eq:lattice-determinant-rules}
\det L_0=|L/L_0|^2\det L,
\end{equation}
where $|L/L_0|$ denotes the number of cosets in $L/L_0$. For an integral lattice $L$, we have $L\subseteq L^*$ and $|L^*/L|=\det L$.

The following lemmas are useful for our argument.

\begin{lemma}\label{lem:integer-inner-products}
Let $L$ be an integral lattice in $\R^n$, and let $H\subseteq L^*/L$ be a finite subgroup. If $\bm v+L\in\operatorname{Syl}_p(H)$ and $\bm w+L\in\operatorname{Syl}_r(H)$ for distinct primes $p,r$, then $\langle\bm v,\bm w\rangle\in\Z$.
\end{lemma}
\begin{proof}
By the definition of the Sylow subgroups, there are integers $a,b\geq0$ such that $p^a\bm v\in L$ and $r^b\bm w\in L$. Denote $t=\langle\bm v,\bm w\rangle$. Since $\bm v,\bm w\in L^*$, both $p^at=\langle p^a\bm v,\bm w\rangle$ and $r^bt=\langle\bm v,r^b\bm w\rangle$ are integers. The distinct primes $p,r$ make $p^a$ and $r^b$ relatively prime. Choose integers $u,v$ with $up^a+vr^b=1$. Then $t=u(p^at)+v(r^bt)\in\Z$.
\end{proof}

\begin{lemma}\label{lem:integral-norm-group}
Let $L$ be an even lattice and let $H\subseteq L^*/L$ be a finite subgroup. Suppose that for any $\bm v+L\in H\setminus\{L\}$, there exists $\bm w+L\in H$ such that $\langle\bm v,\bm w\rangle\notin\Z$. Suppose that $\langle\bm v,\bm v\rangle\in\Z$ for every $\bm v+L\in H$. Then $|H|=4^s$ for some integer $s\geq0$.
\end{lemma}
\begin{proof}
We prove by induction on $|H|$. If $H=\{L\}$, then $|H|=4^s$ with $s=0$. Suppose now that $|H|>1$. For any $\bm a+L,\bm b+L\in H$, we have
\begin{equation}\label{xueq21}
2\langle\bm a,\bm b\rangle
=\langle\bm a+\bm b,\bm a+\bm b\rangle
-\langle\bm a,\bm a\rangle
-\langle\bm b,\bm b\rangle\in\Z.
\end{equation}
Choose $\bm u\in L^*$ with $\bm u+L\in H\setminus\{L\}$. By the first hypothesis there is $\bm v+L\in H$ such that $\langle\bm u,\bm v\rangle\notin\Z$. Equation \eqref{xueq21} then gives $\langle\bm u,\bm v\rangle\in\frac12+\Z$. The cosets $\bm u+L$ and $\bm v+L$ are distinct, since their squared norms are integers. Define four subsets of $H$ by
\[
\begin{aligned}
K_1&=\{\bm a+L\in H:
\langle\bm a,\bm u\rangle\in\Z,
\ \langle\bm a,\bm v\rangle\in\Z\},\\
K_2&=\{\bm a+L\in H:
\langle\bm a,\bm u\rangle\in\tfrac12+\Z,
\ \langle\bm a,\bm v\rangle\in\Z\},\\
K_3&=\{\bm a+L\in H:
\langle\bm a,\bm u\rangle\in\Z,
\ \langle\bm a,\bm v\rangle\in\tfrac12+\Z\},\\
K_4&=\{\bm a+L\in H:
\langle\bm a,\bm u\rangle\in\tfrac12+\Z,
\ \langle\bm a,\bm v\rangle\in\tfrac12+\Z\}.
\end{aligned}
\]
Equation~\eqref{xueq21} shows that every $\bm a+L\in H$ belongs to exactly one of $K_1,K_2,K_3,K_4$. Note that if $\bm a+L\in K_2$, then $\langle\bm a-\bm v,\bm u\rangle\in\Z$ and $\langle\bm a-\bm v,\bm v\rangle\in\Z$, so $\bm a-\bm v+L\in K_1$. In turn, if $\bm a+L\in K_1$, then $\bm a+\bm v+L\in K_2$. This means that $K_2=(\bm v+L)+K_1$. Similarly, we can show that $K_3=(\bm u+L)+K_1$ and $K_4=(\bm u+\bm v+L)+K_1$. Since translation is a bijection on $H$, we have
\begin{equation*}
|K_1|=|K_2|=|K_3|=|K_4|=\frac{|H|}{4}.
\end{equation*}

Note that $K_1$ is a subgroup of $H$. For every $\bm a+L\in K_1$, the squared norm $\langle\bm a,\bm a\rangle$ is an integer because $K_1\subseteq H$. Suppose that $\bm a+L\in K_1$ has integer inner product with every class in $K_1$. By the definition of $K_1$, it also has integer inner product with $\bm u+L$ and $\bm v+L$. The displayed identities for $K_2,K_3,K_4$ then show that $\bm a+L$ has integer inner product with every class in $H$. By the first hypothesis in the lemma, we have $\bm a\in L$. Hence, $K_1$ satisfies this hypothesis as well. Since $|K_1|=|H|/4<|H|$, induction gives $|K_1|=4^{s-1}$ for some integer $s\geq1$. Therefore $|H|=4^s$.
\end{proof}

\begin{definition}\label{def:gauss}
Let $L$ be an even lattice and let $H$ be a subgroup of $L^*/L$. Define its normalized \emph{Gauss sum} by
\begin{equation}\label{eq:gauss}
G_L(H)=\frac1{\sqrt{|H|}}
\sum_{\bm v+L\in H}
\exp\bigl(\pi\ii\langle\bm v,\bm v\rangle\bigr).
\end{equation}
\end{definition}

Let $L$ be an even lattice and let $H$ be a finite subgroup of $L^*/L$. The normalized Gauss sum factors over the primary subgroups, i.e.,
\begin{equation}\label{eq:gauss-primary-product}
G_L(H)=\prod_{p\text{ prime}} G_L(\operatorname{Syl}_p(H)).
\end{equation}
This follows from the orthogonal-sum formula \cite[Section~1.1, p.~206]{TuraevGauss} and the primary decomposition \cite[Section~5.2]{DummitFoote}. Moreover, if an even lattice $L$ has rank $n$, then the Milgram-Braun formula \cite{MH} gives
\begin{equation}\label{eq:milgram}
G_L(L^*/L)=\exp(2\pi\ii n/8).
\end{equation}

The following lemma shows that $G_L(H)$ is real when $p$ is a prime with $p\equiv1\pmod4$. This is the special case needed here of Taylor's result on quadratic Gauss sums over $p$-primary groups; see \cite[Theorem~1.16]{TaylorGauss}. We keep the elementary proof below because it is short and avoids any extra notation. It plays a key role in the proofs of Theorem \ref{thm:main} and Theorem \ref{thm:second}, which allows us to relax the restriction on $m$ obtained in \cite{NV,BMV}.

\begin{lemma}\label{lem:real-gauss}
Let $L$ be an even lattice and $H\subseteq L^*/L$ a subgroup whose order is a power of a prime $p\equiv1\pmod4$. Then $G_L(H)$ is real.
\end{lemma}
\begin{proof}
Choose $a\geq1$ such that $p^a\bm v\in L$ for every $\bm v+L\in H$. Because $L$ is even, $p^{2a}\langle\bm v,\bm v\rangle/2\in\Z$. We claim that there is an integer $u$ such that
\begin{equation}\label{eq:sqrt-minus-one}
u^2\equiv-1\pmod{p^{2a}}
\quad\text{and}\quad
p\nmid u.
\end{equation}
Consider the map $f: H\to H$ defined by $f(\bm v+L)=u\cdot (\bm v+L)$. Since $p\nmid u$, there is an integer $b$ with $bu\equiv1\pmod{p^a}$. Write $bu=1+rp^a$ for an integer $r$. For every $\bm v+L\in H$ we have $p^a(\bm v+L)=L$, so
\[
b\bigl(u(\bm v+L)\bigr)
=u\bigl(b(\bm v+L)\bigr)
=(1+rp^a)(\bm v+L)=\bm v+L.
\]
Then $u(\bm v+L)=u(\bm w+L)$ implies
\[
\bm v+L=b\bigl(u(\bm v+L)\bigr)
=b\bigl(u(\bm w+L)\bigr)=\bm w+L,
\]
so distinct elements have distinct images. Also, every $\bm v+L\in H$ is the image of the element $b(\bm v+L)\in H$, since $u\bigl(b(\bm v+L)\bigr)=\bm v+L$. Hence, the map $f$ is a bijection, and multiplication by $u$ permutes $H$. Write $u^2+1=kp^{2a}$, where $k\in\Z$. Then we have
\[
\exp(\pi\ii u^2\langle\bm v,\bm v\rangle)
=\exp(\pi\ii kp^{2a}\langle\bm v,\bm v\rangle)
\cdot \exp(-\pi\ii\langle\bm v,\bm v\rangle)=\exp(-\pi\ii\langle\bm v,\bm v\rangle),
\]
where the last equation follows from $p^{2a}\langle\bm v,\bm v\rangle/2\in\Z$. Recall that $f$ is a bijection on $H$, so we have
\[
G_L(H)
=\frac1{\sqrt{|H|}}\sum_{\bm v+L\in H}
\exp\bigl(\pi\ii u^2\langle\bm v,\bm v\rangle\bigr)
=\frac1{\sqrt{|H|}}\sum_{\bm v+L\in H}
\exp\bigl(-\pi\ii\langle\bm v,\bm v\rangle\bigr).
\]
Hence, $G_L(H)$ equals its complex conjugate, implying that $G_L(H)\in\R$.

It remains to prove \eqref{eq:sqrt-minus-one}. We give a direct construction of such an integer $u$. Note that
\begin{equation*}
\begin{aligned}
(p-1)! &=\prod_{j=1}^{(p-1)/2}j(p-j)\equiv\prod_{j=1}^{(p-1)/2}j(-j)
\equiv(-1)^{(p-1)/2}   \left(\left(\frac{p-1}{2}\right)!\right)^2 \pmod p. \end{aligned}
\end{equation*}
By Wilson's theorem \cite[Chapter~8]{DudleyWilson}, $(p-1)!\equiv-1\pmod p$. Because $p\equiv1\pmod4$, the exponent $(p-1)/2$ is even. Hence, $u_1=\left(\frac{p-1}{2}\right)!$ satisfies $u_1^2\equiv (p-1)!\equiv -1\pmod p$ and $p\nmid u_1$. Suppose inductively that $u_k^2+1=c_kp^k$ for some $k\geq1$ with $p\nmid u_k$. Since $2u_k$ is invertible modulo $p$, choose an integer $r_k$ satisfying $c_k+2u_kr_k\equiv0\pmod p$ and set $u_{k+1}=u_k+r_kp^k$. Then
\[
u_{k+1}^2+1
=p^k(c_k+2u_kr_k)+r_k^2p^{2k}\equiv 0\pmod {p^{k+1}}.
\]
Moreover $u_{k+1}\equiv u_k\pmod p$, so $p\nmid u_{k+1}$. After $2a-1$ such steps, $u=u_{2a}$ satisfies \eqref{eq:sqrt-minus-one}.

\end{proof}

\section{The lattices generated by tight spherical $5$-designs}\label{sec:lattice}

In this section we study the lattice generated by a tight spherical $5$-design. We retain the notation of Bannai, Munemasa, and Venkov \cite[Sections~2 and~3.1]{BMV}. Throughout this paper, we denote
\begin{equation}\label{eq:parameters}
m\in\Z_{>0},\qquad d=2m+1,\qquad
n=d^2-2=4m(m+1)-1.
\end{equation}
Write
\[
\mathbb{S}^{n-1} (d)=\{\bm{x}\in\R^n:\langle\bm{x},\bm{x}\rangle=d\}.
\]
A finite subset $A$ on $\mathbb{S}^{n-1} (d)$ is called a spherical $t$-design if $\frac{1}{\sqrt{d}}A$ is a spherical $t$-design in the unit sphere. A tight spherical $5$-design on $\mathbb{S}^{n-1} (d)$ can be scaled and written as $D=X\cup(-X)\subset \mathbb{S}^{n-1} (d)$, where $X\cap -X=\varnothing$, and
\begin{equation}\label{eq:scaled}
|X|=\frac{n(n+1)}2,\qquad
\langle\bm{x},\bm{x}\rangle=d,\qquad
\langle\bm{x},\bm{y}\rangle\in\{1,-1\}
\quad\forall\bm{x},\bm{y}\in X, \bm x\ne\bm y.
\end{equation}
For every $\bm{\alpha}\in\R^n$, we have \cite{NV,BMV}
\begin{align}
\sum_{\bm{x}\in X}\langle\bm{x},\bm{\alpha}\rangle^2
&=2m(m+1)d\langle\bm{\alpha},\bm{\alpha}\rangle,
\label{eq:S2}\\
\sum_{\bm{x}\in X}\langle\bm{x},\bm{\alpha}\rangle^4
&=6m(m+1)\langle\bm{\alpha},\bm{\alpha}\rangle^2.
\label{eq:S4}
\end{align}

Throughout this section, suppose that a tight spherical $5$-design on $\mathbb{S}^{n-1} (d)$ exists with the parameters in \eqref{eq:parameters}. Following \cite[equation (20)-(24)]{BMV}, we define
\begin{equation}\label{eq:lattices}
\begin{split}
\Lambda&=\left\{\sum_{\bm{x}\in X}c_{\bm{x}}\bm{x}:
c_{\bm{x}}\in\Z\right\},\\
\Lambda_+&=\left\{\sum_{\bm{x}\in X}c_{\bm{x}}\bm{x}:
c_{\bm{x}}\in\Z, \sum_{\bm{x}\in X}c_{\bm{x}}\equiv 0\pmod2 \right\},\\
\Gamma&=\frac1{\sqrt2}\Lambda_+.
\end{split}
\end{equation}
Note that for any $\bm{\alpha}\in\R^n$, we have $\bm{\alpha}\in\Lambda^*$ if and only if $\langle\bm{\alpha},\bm{x}\rangle\in\Z$ for every $\bm{x}\in X$.

The following lemma strengthens the scalar integrality condition \cite[Lemma 3.3]{BMV}.

\begin{lemma}\label{lem:J}
For every $\bm{\alpha}\in\Lambda^*$, we have
\begin{equation}\label{eq:J}
\frac{m(m+1)\langle\bm{\alpha},\bm{\alpha}\rangle(3\langle\bm{\alpha},\bm{\alpha}\rangle-d)}{12}\in\Z.
\end{equation}
\end{lemma}
\begin{proof}

There is nothing to prove when $\bm{\alpha}=\bm{0}$. Assume $\bm{\alpha}\neq \bm{0}$. Choose $\bm{x}_0\in X$ and set $\bm{\beta}=\bm{\alpha}+\frac{1}{2}\bm{x}_0$. Denote
\begin{equation*}
u=\langle\bm{x}_0,\bm{\alpha}\rangle\in\Z,
\quad
t=\langle\bm{\alpha},\bm{\alpha}\rangle
\quad
\text{and}\quad
I:=\frac{1}{24}	\sum_{\bm{x}\in X} \left(\langle\bm{x},\bm{\beta}\rangle^2-\frac{1}{4}\right)\cdot \left(\langle\bm{x},\bm{\beta}\rangle^2-\frac{9}{4}\right).
\end{equation*}
Note that $\langle\bm{x},\bm{\beta}\rangle$ belongs to $\frac12+\Z$ for each $\bm x\in X$, because $\langle\bm{x},\bm{x}_0\rangle$ is odd. Since four consecutive integers have product divisible by $24$, we see that for every $k\in\Z$,
\[
\frac{1}{24}\cdot \left((k+\frac{1}{2})^2-\frac{1}{4}\right)\cdot
\left((k+\frac{1}{2})^2-\frac{9}{4}\right)
=\frac{1}{24} (k-1)k(k+1)(k+2)\in\Z,
\]
Therefore, we have $I\in\Z$. On the other hand, the moment identities \eqref{eq:S2} and \eqref{eq:S4} give
\begin{equation*}
\begin{aligned}
I&=\frac{1}{24}	\sum_{\bm{x}\in X} \langle\bm{x},\bm{\beta}\rangle^4
-\frac{5}{48} \sum_{\bm{x}\in X} \langle\bm{x},\bm{\beta}\rangle^2+\frac{1}{24}\cdot \frac{9}{16}\cdot |X|\\
&=\frac{1}{24}\cdot 6m(m+1)\langle\bm{\beta},\bm{\beta}\rangle^2-\frac{5}{48}\cdot 2m(m+1)d \langle\bm{\beta},\bm{\beta}\rangle+\frac{1}{24}\cdot \frac{9}{16}\cdot |X|\\
&\overset{(a)}=\frac{m(m+1)t(3t-d)}{12}
+I_1+I_2+I_3,
\end{aligned}
\end{equation*}
where
\begin{equation*}
I_1:=\frac{m(m+1)tu}{2},
\quad I_2:=\frac{m(m+1)u(3u-d)}{12},
\quad I_3:=\frac{(m-1)m(m+1)(m+2)}{24}.
\end{equation*}
Here, equation ($a$) follows from the substitution $\langle\bm{\beta},\bm{\beta}\rangle=t+u+\frac{d}{4}$ and $|X|=\frac12 n(n+1)=2m(m+1)(d^2-2)$. We claim that
\begin{equation}\label{eq:J2}
m(m+1)t\in2\Z.
\end{equation}
Then we have $I_1\in \Z$. Note that both $m(m+1)$ and $u(3u-d)$ are even, so $m(m+1)u(3u-d)$ is divisible by $4$. Note that $m(m+1)u(3u-d)$ is also divisible by $3$ because we have either $3\mid m(m+1)$, or $m\equiv1\pmod3$ and $3u-d=3u-2m-1 \equiv0\pmod3$. Hence, $I_2$ is also an integer. Since $(m-1)m(m+1)(m+2)$ is divisible by $24$, $I_3$ is an integer. Since $I, I_1,I_2,I_3\in \Z$, we conclude that $\frac{1}{12}m(m+1)t(3t-d)$ is an integer. We arrive at \eqref{eq:J}.

It remains to prove \eqref{eq:J2}. Write $t=p/q$ with $p,q\in\Z$, $q>0$, and $\gcd(p,q)=1$. The fourth-moment identity \eqref{eq:S4} gives $q^2\mid6m(m+1)$. For each prime $r\geq5$ this implies $2\nu_r(q)\leq\nu_r(m(m+1))$, and for $r=2,3$ it implies $2\nu_r(q)\leq1+\nu_r(m(m+1))$. Hence, for each prime $r\geq 2$, we have $\nu_r(m(m+1))\geq \nu_r(q)$, implying that $q\mid m(m+1)$. Therefore, we have $m(m+1)t\in\Z$. If $m(m+1)t$ is an odd integer, then the integrality of
\[
\sum_{\bm{x}\in X}\langle\bm{x},\bm{\alpha}\rangle^4
=\frac{6\bigl(m(m+1)t\bigr)^2}{m(m+1)}
\]
forces $\nu_2(m(m+1))=1$, so the sum $\sum_{\bm{x}\in X}\langle\bm{x},\bm{\alpha}\rangle^4$ becomes an odd integer. Note that $z^4\equiv z^2\pmod2$ for every integer $z$. Hence, the second-moment sum $\sum_{\bm{x}\in X}\langle\bm{x},\bm{\alpha}\rangle^2$ is also an odd integer. However, \eqref{eq:S2} makes the second-moment sum $\sum_{\bm{x}\in X}\langle\bm{x},\bm{\alpha}\rangle^2$ an even integer. We arrive at a contradiction. Hence, $m(m+1)t$ is even. This completes the proof.

\end{proof}

In the following lemma, we present several basic properties of the lattices $\Lambda$ and $\Gamma$. In particular, Lemma \ref{lem:even} (ii) shows that $\Gamma$ is an even lattice for any $m>0$, which removes the extra hypothesis $8\nmid m(m+1)$ from the evenness assertion in \cite[Lemma 3.7]{BMV}.

\begin{lemma}\label{lem:even}
The following holds.
\begin{enumerate}

	\item[(i)]
For each $\bm x_0\in X$, we have $|\Lambda/\Lambda_+|=2$ and
\begin{equation}\label{xueq35}
(\Lambda_+)^*=\Lambda^*\sqcup\left(\frac{\bm x_0}{2}+\Lambda^*\right).
\end{equation}

\item[(ii)] The lattice $\Gamma=\frac{1}{\sqrt{2}}\Lambda_+$ is even for every $m>0$.
Moreover, we have
\begin{equation}\label{xueq11}
\Gamma\subset \sqrt{2}\Lambda^*\subset \Gamma^*.
\end{equation}

\item[(iii)] The quotient $\Gamma^*/\Gamma$ has the disjoint decomposition
\[
\Gamma^*/\Gamma
=\sqrt{2}\Lambda^*/\Gamma\sqcup
\left\{\frac{\bm{x}_0}{\sqrt2}+\sqrt2\bm\alpha+\Gamma:
\bm\alpha\in\Lambda^*\right\}.
\]
Consequently, we have $|\Gamma^*/\Gamma|=2\cdot |\sqrt{2}\Lambda^*/\Gamma|$ and $\det\Lambda=2^{n-1}\cdot |\sqrt{2}\Lambda^*/\Gamma|$.

\item[(iv)] We have
\begin{equation}\label{xueq12}
\begin{aligned}
G_\Gamma(\Gamma^*/\Gamma)
&=G_\Gamma(\sqrt{2}\Lambda^*/\Gamma)\,
\frac{1+\exp(\pi\ii d/2)}{\sqrt2}.
\end{aligned}
\end{equation}
\end{enumerate}

\end{lemma}
\begin{proof}
(i) Fix $\bm x_0\in X$. For any expression $\bm\lambda=\sum_{\bm x\in X}c_{\bm x}\bm x\in\Lambda$, all the numbers $\langle\bm x,\bm x_0\rangle$ are odd by \eqref{eq:scaled}. Hence
\[
\langle\bm\lambda,\bm x_0\rangle
\equiv\sum_{\bm x\in X}c_{\bm x}\pmod2.
\]
The left side depends only on $\bm\lambda$, so the parity of the coefficient sum is independent of the expression. By the definition of $\Lambda_+$, we obtain
\begin{equation}\label{xueq30}
\Lambda_+=\{\bm\lambda\in\Lambda:
\langle\bm\lambda,\bm x_0\rangle\in2\Z\}.
\end{equation}
Moreover $\langle\bm x_0,\bm x_0\rangle=d$ is odd, so $\bm x_0\notin\Lambda_+$ and $\Lambda=\Lambda_+\sqcup(\bm x_0+\Lambda_+)$. This proves $|\Lambda/\Lambda_+|=2$.

It remains to prove \eqref{xueq35}. The displayed description of $\Lambda_+$ gives $\bm x_0/2\in(\Lambda_+)^*$, whereas $\bm x_0/2\notin\Lambda^*$ because $d/2\notin\Z$. To see that these are all the classes in $(\Lambda_+)^*$, let $\bm\beta\in(\Lambda_+)^*$. Since $2\bm x_0\in\Lambda_+$, $\langle\bm\beta,\bm x_0\rangle\in\frac12\Z$. If this is an integer, then $\bm\beta$ pairs integrally with both $\Lambda_+$ and $\bm x_0$. Combining with $\Lambda=\Lambda_+\sqcup(\bm x_0+\Lambda_+)$, we see that $\bm\beta\in\Lambda^*$. Otherwise, $\langle\bm\beta-\bm x_0/2,\bm x_0\rangle\in\Z$ because $d$ is odd. Since $\bm\beta-\bm x_0/2\in(\Lambda_+)^*$,  the same argument makes $\bm\beta-\bm x_0/2$ in $\Lambda^*$, i.e., $\bm\beta\in \bm x_0/2+\Lambda^*$. Since $\bm x_0/2+\Lambda^*$ and $\Lambda^*$ are disjoint, we arrive at \eqref{xueq35}.

(ii) Let $\bm\lambda=\sum_{i=1}^Nc_i\bm x_i\in\Lambda_+$, where $X=\{\bm x_1,\ldots,\bm x_N\}$ and $S=\sum_{i=1}^N c_i$ is even. By \eqref{eq:scaled},
\[
\langle\bm\lambda,\bm\lambda\rangle
\equiv d\sum_{i=1}^N c_i^2
+2\sum_{1\leq i<j\leq N}c_ic_j
= (d-1)\sum_{i=1}^N c_i^2+S^2\pmod4.
\]
Here $\sum_{i=1}^N c_i^2\equiv S\equiv0\pmod2$ and $d-1=2m$. Both terms on the right are divisible by $4$. Hence, $\langle\bm\lambda/\sqrt2, \bm\lambda/\sqrt2\rangle\in2\Z$, so $\Gamma$ is even. For every $\bm x\in X$, the parity calculation from (i), with $\bm x$ in place of $\bm x_0$, gives $\langle\bm\lambda,\bm x\rangle\in2\Z$. As $X$ generates $\Lambda$, this implies $\bm\lambda/2\in\Lambda^*$ and hence $\Gamma\subseteq\sqrt2\Lambda^*$. Finally, for $\bm\alpha\in\Lambda^*$ and $\bm\lambda\in\Lambda_+$, $\langle\sqrt2\bm\alpha,\bm\lambda/\sqrt2\rangle =\langle\bm\alpha,\bm\lambda\rangle\in\Z$. Therefore $\sqrt2\Lambda^*\subseteq\Gamma^*$.

(iii) The definition of the dual lattice and $\Gamma=\Lambda_+/\sqrt2$ give $\Gamma^*=\sqrt2(\Lambda_+)^*$. Multiply the disjoint decomposition in (i) by $\sqrt2$ and take classes modulo $\Gamma\subseteq\sqrt2\Lambda^*$ from (ii). This gives the stated disjoint decomposition of $\Gamma^*/\Gamma$. Its two parts have the same cardinality, so $|\Gamma^*/\Gamma|=2|\sqrt2\Lambda^*/\Gamma|$. Since $|\Lambda/\Lambda_+|=2$ and $\Gamma$ has rank $n$, the index and scaling formulas give
\[
\det\Gamma=2^{-n}\det\Lambda_+
=2^{2-n}\det\Lambda.
\]
The lattice $\Gamma$ is integral by (ii), whence $\det\Gamma=|\Gamma^*/\Gamma|$. Combining these equalities proves $\det\Lambda=2^{n-1}|\sqrt2\Lambda^*/\Gamma|$.

(iv) Choose one $\bm\alpha\in\Lambda^*$ for each class $\sqrt2\bm\alpha+\Gamma\in\sqrt2\Lambda^*/\Gamma$. By (iii), addition of $\bm x_0/\sqrt2+\Gamma$ maps these classes bijectively onto the other half of $\Gamma^*/\Gamma$. The class paired with $\sqrt2\bm\alpha+\Gamma$ is represented by $\bm x_0/\sqrt2+\sqrt2\bm\alpha$. Its squared norm is
\[
\frac d2+2\langle\bm\alpha,\bm\alpha\rangle
+2\langle\bm x_0,\bm\alpha\rangle.
\]
The last inner product is an integer because $\bm\alpha\in\Lambda^*$ and $\bm x_0\in\Lambda$. The exponential for the first class is $\exp(2\pi\ii\langle\bm\alpha,\bm\alpha\rangle)$. For the second class, the extra factor $\exp(2\pi\ii\langle\bm x_0,\bm\alpha\rangle)$ is $1$, so its exponential is the first one multiplied by $\exp(\pi\ii d/2)$. Since $\Gamma$ is even by (ii), these exponentials do not depend on the chosen representatives. Part~(iii) gives $|\Gamma^*/\Gamma|=2|\sqrt2\Lambda^*/\Gamma|$. Hence, the definition \eqref{eq:gauss} yields
\[
\begin{aligned}
G_\Gamma(\Gamma^*/\Gamma)
&=\frac{1}{\sqrt{2|\sqrt2\Lambda^*/\Gamma|}}
\sum_{\sqrt2\bm\alpha+\Gamma\in\sqrt2\Lambda^*/\Gamma}
\exp\bigl(2\pi\ii\langle\bm\alpha,\bm\alpha\rangle\bigr)
\bigl(1+\exp(\pi\ii d/2)\bigr)\\
&=\frac{1+\exp(\pi\ii d/2)}{\sqrt2}\,
G_\Gamma(\sqrt2\Lambda^*/\Gamma),
\end{aligned}
\]
which is \eqref{xueq12}.

\end{proof}

\begin{lemma}\label{lem:phase}
If $m$ is even, then
\begin{equation}\label{eq:required-phase}
G_\Gamma(\sqrt{2}\Lambda^*/\Gamma)=-\ii.
\end{equation}
\end{lemma}
\begin{proof}
By Lemma \ref{lem:even} (ii), we see that $\Gamma$ is even. Since $m(m+1)$ is even, $n=4m(m+1)-1\equiv7\pmod8$, so Milgram's formula \eqref{eq:milgram} gives $G_\Gamma(\Gamma^*/\Gamma)=(1-\ii)/\sqrt2$. On the other hand, since $m$ is even, $d=2m+1\equiv1\pmod4$ and $\exp(\pi\ii d/2)=\ii$. Substituting $\exp(\pi\ii d/2)=\ii$ in \eqref{xueq12} therefore yields $G_\Gamma(\sqrt{2}\Lambda^*/\Gamma)=(1-\ii)/(1+\ii)=-\ii$.
\end{proof}

In the following lemmas, we consider the Sylow $p$-subgroup $\operatorname{Syl}_p(\sqrt2\Lambda^*/\Gamma)$ of $\sqrt2\Lambda^*/\Gamma$ for a given prime $p$. Recall that
\[
\operatorname{Syl}_p(\sqrt2\Lambda^*/\Gamma)
=\{\sqrt2\bm\alpha+\Gamma\in\sqrt2\Lambda^*/\Gamma:
p^a\cdot \sqrt2\bm\alpha\in \Gamma
\text{ for some integer }a\geq0\}.
\]

\begin{lemma}\label{lem:norm-denominators}
Let $p$ be a prime and let $\sqrt2\bm\alpha+\Gamma\in\operatorname{Syl}_p(\sqrt{2}\Lambda^*/\Gamma)$. The following assertions hold.
\begin{enumerate}
\item[(i)] If $\langle\bm\alpha,\bm\alpha\rangle=\frac{u}{v}\neq 0$, where $u,v$ are relatively prime positive integers, then $v$ is a power of $p$.
Moreover, in this case,
\begin{equation}\label{eq:norm-valuation}
2\nu_p(\langle\bm\alpha,\bm\alpha\rangle)+\nu_p(6m(m+1))\geq0.
\end{equation}

\item[(ii)] If $\sqrt2\bm\alpha+\Gamma\ne\Gamma$,
then there is a class $\sqrt2\bm\beta+\Gamma\in\operatorname{Syl}_p(\sqrt{2}\Lambda^*/\Gamma)$ such that $2\langle\bm\alpha,\bm\beta\rangle\notin\Z$.

\item[(iii)] If $\langle\bm\beta,\bm\beta\rangle\in\Z$ for every
$\sqrt2\bm\beta+\Gamma\in\operatorname{Syl}_p(\sqrt{2}\Lambda^*/\Gamma)$, then $\operatorname{Syl}_p(\sqrt{2}\Lambda^*/\Gamma)=\{\Gamma\}$.
\end{enumerate}
\end{lemma}
\begin{proof}
(i) For some $a\geq0$, $p^a\sqrt2\bm\alpha\in\Gamma$. Since $\Gamma$ is even, its squared norm $2p^{2a}\langle\bm\alpha,\bm\alpha\rangle$ is an even integer. Hence, $p^{2a}\langle\bm\alpha,\bm\alpha\rangle=\frac{p^{2a} u}{v}\in\Z$, implying that $v$ is a power of $p$. The fourth-moment identity \eqref{eq:S4} is a sum of integer fourth powers when $\bm\alpha\in\Lambda^*$. Hence, we have $6m(m+1)\langle\bm\alpha,\bm\alpha\rangle^2\in\Z$, which gives \eqref{eq:norm-valuation}.

(ii) We prove by contradiction. Let $\sqrt{2}\bm{\alpha}+\Gamma\in\operatorname{Syl}_p(\sqrt{2}\Lambda^*/\Gamma)\setminus\{\Gamma\}$. Assume that $2\langle\bm\alpha,\bm{\alpha'} \rangle\in\Z$ for any $\sqrt2\bm\alpha'+\Gamma\in\operatorname{Syl}_p(\sqrt{2}\Lambda^*/\Gamma)$. Take any $\bm\beta\in\Lambda^*$ and decompose $\sqrt{2}\bm{\beta}+\Gamma$ into its primary components. By Lemma \ref{lem:even} (ii), we have $\Gamma\subseteq\sqrt2\Lambda^*\subseteq\Gamma^*$. Applying Lemma \ref{lem:integer-inner-products} with $L=\Gamma$ and $H=\sqrt2\Lambda^*/\Gamma$ shows that $\sqrt2\bm\alpha$ has integer inner product with every representative of every component of $\sqrt2\bm\beta+\Gamma$ in $\operatorname{Syl}_r(\sqrt2\Lambda^*/\Gamma)$ with $r\ne p$. Since $\sqrt{2}\bm{\alpha}+\Gamma$ also has integer inner product with the prime component of $\sqrt{2}\bm{\beta}+\Gamma$ in $\operatorname{Syl}_{p}(\sqrt{2}\Lambda^*/\Gamma)$, we have $2\langle\bm\alpha,\bm\beta\rangle\in\Z$ for every $\bm\beta\in\Lambda^*$. Then $2\bm\alpha\in(\Lambda^*)^*=\Lambda$. Fix $\bm x_0\in X$. Since $\langle2\bm\alpha,\bm x_0\rangle =2\langle\bm\alpha,\bm x_0\rangle\in2\Z$, by \eqref{xueq30} we have $2\bm\alpha\in\Lambda_+$. Consequently, $\sqrt2\bm\alpha\in\Gamma$. This contradicts $\sqrt{2}\bm{\alpha}+\Gamma\in\operatorname{Syl}_p(\sqrt{2}\Lambda^*/\Gamma)\setminus\{\Gamma\}$. Hence, we prove the lemma.

(iii) Take any $\sqrt2\bm\alpha+\Gamma, \sqrt2\bm\beta+\Gamma\in \operatorname{Syl}_p(\sqrt{2}\Lambda^*/\Gamma)$. Their sum $\sqrt2\bm\alpha+\sqrt2\bm\beta+\Gamma$ also lies in $\operatorname{Syl}_p(\sqrt{2}\Lambda^*/\Gamma)$. Hence,
\[
2\langle\bm\alpha,\bm\beta\rangle
=\langle\bm\alpha+\bm\beta,\bm\alpha+\bm\beta\rangle
-\langle\bm\alpha,\bm\alpha\rangle
-\langle\bm\beta,\bm\beta\rangle\in\Z.
\]
By (ii) we must have $\operatorname{Syl}_p(\sqrt{2}\Lambda^*/\Gamma)=\{\Gamma\}$.
\end{proof}

\begin{lemma}\label{lem:local}
The following assertions hold.
\begin{enumerate}
\item[(i)] If $\nu_2(m(m+1))\leq5$, then every
$\sqrt2\bm{\alpha}+\Gamma\in\operatorname{Syl}_{2}(\sqrt{2}\Lambda^*/\Gamma)$ satisfies $2\langle\bm{\alpha},\bm{\alpha}\rangle\in\Z$. Moreover, $\nu_2(\det\Lambda)$ is even.
\item[(ii)] If $p\geq5$ is prime and $\nu_p(m(m+1))\leq1$,
then $\operatorname{Syl}_p(\sqrt{2}\Lambda^*/\Gamma)=\{\Gamma\}$.
\end{enumerate}
\end{lemma}
\begin{proof}
(i) Suppose $\nu_2(m(m+1))\leq5$, and take $\sqrt2\bm\alpha+\Gamma\in\operatorname{Syl}_2(\sqrt2\Lambda^*/\Gamma)$. If $\bm\alpha=\bm0$, then $2\langle\bm\alpha,\bm\alpha\rangle=0\in\Z$. Now suppose $\bm\alpha\ne\bm0$. By Lemma \ref{lem:norm-denominators} (i), the reduced denominator of $\langle\bm\alpha,\bm\alpha\rangle$ is a power of $2$. If $\nu_2(\langle\bm\alpha,\bm\alpha\rangle)\geq0$, this denominator is $1$, so $\langle\bm\alpha,\bm\alpha\rangle\in\Z$ and hence $2\langle\bm\alpha,\bm\alpha\rangle\in\Z$. It remains to consider $\nu_2(\langle\bm\alpha,\bm\alpha\rangle)=-h<0$. Then $\nu_2(3\langle\bm\alpha,\bm\alpha\rangle-d)=-h$ since $d=2m+1$ is odd. The integer in \eqref{eq:J} has $2$-adic valuation
\[
\nu_2(m(m+1))-2h-2\geq0.
\]
It follows that $2h\leq3$, so $h=1$. The reduced denominator of $\langle\bm\alpha,\bm\alpha\rangle$ is then $2$, and therefore $2\langle\bm\alpha,\bm\alpha\rangle\in\Z$.

It remains to prove that $\nu_2(\det\Lambda)$ is even. Lemma \ref{lem:even} (ii) shows that $\Gamma$ is even and identifies $\operatorname{Syl}_{2}(\sqrt{2}\Lambda^*/\Gamma)$ as a subgroup of $\Gamma^*/\Gamma$. Combining with Lemma \ref{lem:norm-denominators} (ii), we can apply Lemma \ref{lem:integral-norm-group} with $L=\Gamma$ and $H=\operatorname{Syl}_{2}(\sqrt{2}\Lambda^*/\Gamma)$ to obtain that $|\operatorname{Syl}_{2}(\sqrt{2}\Lambda^*/\Gamma)|=4^s=2^{2s}$ for some $s\geq0$. Finally, Lemma \ref{lem:even} (iii) gives $\nu_2(\det\Lambda)=n-1+2s$, which is even because $n=4m(m+1)-1$ is odd.

(ii) Suppose that $p\geq5$ and $\nu_p(m(m+1))\leq1$. For a class $\sqrt2\bm{\alpha}+\Gamma\in\operatorname{Syl}_p(\sqrt{2}\Lambda^*/\Gamma)$, put $t=\langle\bm\alpha,\bm\alpha\rangle$. If $t=0$, then $t\in\Z$. If $t\ne0$, Lemma \ref{lem:norm-denominators} (i) makes its reduced denominator a power of $p$. Equation \eqref{eq:norm-valuation} and $\nu_p(6m(m+1))\leq1$ imply $\nu_p(t)\geq0$. Thus this denominator is $1$, and again $t\in\Z$. Lemma \ref{lem:norm-denominators} (iii) now gives $\operatorname{Syl}_p(\sqrt{2}\Lambda^*/\Gamma)=\{\Gamma\}$.

\end{proof}

\section{Proof of Theorem~\ref{thm:main}}\label{sec:first}
\begin{lemma}\label{lem:three-trivial}
If $3\nmid m(m+1)$, then $\operatorname{Syl}_{3}(\sqrt{2}\Lambda^*/\Gamma)=\{\Gamma\}$.
\end{lemma}
\begin{proof}
Let $\sqrt2\bm\alpha+\Gamma$ belong to $\operatorname{Syl}_{3}(\sqrt{2}\Lambda^*/\Gamma)$. We show that $\langle\bm\alpha,\bm\alpha\rangle\in\Z$. This is immediate if $\bm\alpha=\bm0$. If $\bm\alpha\ne\bm0$, write $\langle\bm\alpha,\bm\alpha\rangle=\frac{u}{v}$, where $u,v$ are relatively prime positive integers. By Lemma \ref{lem:norm-denominators} (i), $v$ is a power of $3$. Equation \eqref{eq:norm-valuation} and $\nu_3(6m(m+1))=1$ give $2\nu_3(\langle\bm\alpha,\bm\alpha\rangle)+1\geq0$. As $\nu_3(\langle\bm\alpha,\bm\alpha\rangle)$ is an integer, $\nu_3(\langle\bm\alpha,\bm\alpha\rangle)\geq0$. Consequently, $v=1$, so $\langle\bm\alpha,\bm\alpha\rangle\in\Z$ in this case as well. Applying Lemma \ref{lem:norm-denominators} (iii), we obtain $\operatorname{Syl}_{3}(\sqrt{2}\Lambda^*/\Gamma)=\{\Gamma\}$.
\end{proof}

Now we can give a proof of Theorem~\ref{thm:main}.

\begin{proof}[Proof of Theorem~\ref{thm:main}]
Suppose that the asserted design exists. By \eqref{eq:gauss-primary-product},
\begin{equation}\label{xueq34}
G_\Gamma(\sqrt{2}\Lambda^*/\Gamma)
=\prod_{p\text{ prime}}G_\Gamma(\operatorname{Syl}_p(\sqrt{2}\Lambda^*/\Gamma)).
\end{equation}
For every $\sqrt2\bm\alpha+\Gamma\in\operatorname{Syl}_{2}(\sqrt{2}\Lambda^*/\Gamma)$, Lemma \ref{lem:local} (i) gives $2\langle\bm\alpha,\bm\alpha\rangle\in\Z$. Hence, every summand defining $G_\Gamma(\operatorname{Syl}_{2}(\sqrt{2}\Lambda^*/\Gamma))$ is $1$ or $-1$, so $G_\Gamma(\operatorname{Syl}_{2}(\sqrt{2}\Lambda^*/\Gamma))\in\R$. Since $m\equiv1\pmod3$, we have $3\nmid m(m+1)$. By Lemma \ref{lem:three-trivial}, we have $\operatorname{Syl}_{3}(\sqrt{2}\Lambda^*/\Gamma)=\{\Gamma\}$, so $G_\Gamma(\operatorname{Syl}_3(\sqrt{2}\Lambda^*/\Gamma))=1$. For each prime $p\geq7$ with $p\equiv3\pmod4$, since $\nu_p(m(m+1))\leq1$, Lemma \ref{lem:local} (ii) shows that $\operatorname{Syl}_p(\sqrt{2}\Lambda^*/\Gamma)=\{\Gamma\}$ and hence $G_\Gamma(\operatorname{Syl}_p(\sqrt{2}\Lambda^*/\Gamma))=1$. For each prime $p$ with $p\equiv1\pmod4$, Lemma \ref{lem:real-gauss} shows that $G_\Gamma(\operatorname{Syl}_p(\sqrt{2}\Lambda^*/\Gamma))$ is real for each such primary subgroup. Putting all these together, we can obtain from \eqref{xueq34} that $G_\Gamma(\sqrt{2}\Lambda^*/\Gamma)$ is real. However, this contradicts Lemma \ref{lem:phase}, which shows that $G_\Gamma(\sqrt{2}\Lambda^*/\Gamma)=-\ii$. Hence, such a tight spherical $5$-design does not exist.
\end{proof}

\section{Proof of Theorem~\ref{thm:second}}\label{sec:second}
In this section, we prove Theorem~\ref{thm:second} by computing the residue of $(\det\Lambda)/3$ modulo $3$ in two ways. Assuming that such a design exists, we first use the Gauss-sum identities to show that $\operatorname{Syl}_3(\sqrt2\Lambda^*/\Gamma)\ne\{\Gamma\}$. Lemma \ref{lem:three-primary} (ii) then gives $\nu_3(\det\Lambda)=1$. Lemma~\ref{lem:determinant-residue} then gives $\frac{\det\Lambda}{3}\equiv 2\pmod 3$. The restrictions on prime divisors, together with Lemma \ref{lem:local} (i),(ii), instead give $\frac{\det\Lambda}{3}\equiv 1\pmod 3$. These residues contradict each other, so such a design does not exist.

\subsection{Lemmas}

We first establish several lemmas using only elementary linear algebra. Their proofs are postponed to the appendix.

The following statement holds over any field $\F$, including fields of characteristic $2$. It will be applied with $\F=\F_3:=\Z/3\Z$. The proof uses only a change of basis and a determinant calculation.

\begin{lemma}\label{lem:constant-pairing-determinant}
Let $s\geq2$ be an integer, and let $\bm D$ be an invertible symmetric matrix of order $2s-1$ over a field $\F$. Suppose that $\bm z_1,\ldots,\bm z_{s}\in\F^{2s-1}$ are linearly independent and satisfy
\begin{equation}\label{eq:constant-pairing-hypothesis}
\bm z_i^{\mathsf T}\bm D\bm z_j=1
\qquad(1\leq i,j\leq s).
\end{equation}
There exists $a\in\F\setminus\{0\}$ such that $\det\bm D=(-1)^{s-1} a^2$.
\end{lemma}
\begin{proof}
See Appendix~\ref{app:constant-pairing-determinant}.
\end{proof}

The following construction gives orthogonal vectors using only integer linear combinations, with change-of-basis determinant not divisible by $3$. It may also be obtained from the Jordan decomposition of lattices over valuation rings (see Zemel \cite[Proposition~1.2 and Corollary~1.3]{ZemelValuation}). 
Here, we include a short direct proof.

\begin{lemma}\label{lem:rational-diagonalization}
Let $\bm B\in\mathbb{Z}^{n\times n}$ be a nonsingular symmetric integral matrix. There is an integral matrix $\bm P\in\mathbb{Z}^{n\times n}$ with $3\nmid\det\bm P$ such that
\begin{equation}\label{eq:rational-diagonalization}
\bm P^{\mathsf T}\bm B\bm P
=\operatorname{diag} (3^{a_1}u_1,\ldots,3^{a_n}u_n),
\qquad a_j\in\Z_{\geq 0},\quad u_j\in\Z,\quad 3\nmid u_j.
\end{equation}
\end{lemma}
\begin{proof}

See Appendix~\ref{app:rational-diagonalization}.

\end{proof}

For an integer matrix $\bm C=(\bm c_1\ \cdots\ \bm c_t)$, define $\operatorname{rank}_{\F_3}\bm C$ as the largest number $\rho$ of distinct columns $\bm c_{j_1},\ldots,\bm c_{j_\rho}$ such that, for all integers $a_1,\ldots,a_\rho$,
\[
\sum_{i=1}^{\rho}a_i\bm c_{j_i}\equiv\bm0\pmod3
\quad\Longrightarrow\quad
a_1\equiv\cdots\equiv a_\rho\equiv0\pmod3.
\]

This rank estimate is the standard consequence of the Smith normal form over $\Z$ \cite[Theorem~2.1]{StanleySNF}. We give a short direct proof used here.

\begin{lemma}\label{lem:rank-reduction-three}
Let $\bm C$ be a nonsingular integral matrix of order $t$. Then
\begin{equation}\label{eq:rank-reduction-three}
\operatorname{rank}_{\F_3}\bm C\geq t-\nu_3(\det\bm C).
\end{equation}
\end{lemma}
\begin{proof}
See Appendix~\ref{app:rank-reduction-three}.

\end{proof}

\subsection{Squared norms in the \texorpdfstring{$3$}{3}-primary subgroup}
\begin{lemma}\label{lem:three-primary}
Suppose that $m\equiv0\pmod3$ and $\nu_3(m(m+1))=1$. The following assertions hold.
\begin{enumerate}
\item[(i)] For every nonidentity class
$\sqrt2\bm\alpha+\Gamma\in\operatorname{Syl}_{3}(\sqrt{2}\Lambda^*/\Gamma)$, we have $\langle\bm\alpha,\bm\alpha\rangle\in\tfrac13+\Z$.
\item[(ii)] If
$\operatorname{Syl}_{3}(\sqrt{2}\Lambda^*/\Gamma)\ne\{\Gamma\}$, then $|\operatorname{Syl}_{3}(\sqrt{2}\Lambda^*/\Gamma)|=3$ and $\nu_3(\det\Lambda)=1$.
\end{enumerate}
\end{lemma}
\begin{proof}
(i) Let $\sqrt2\bm\alpha+\Gamma\in\operatorname{Syl}_{3}(\sqrt{2}\Lambda^*/\Gamma)$ and denote $t=\langle\bm\alpha,\bm\alpha\rangle$. Lemma \ref{lem:norm-denominators} (i) makes the reduced denominator of $t$ a power of $3$ when $t\ne0$. Since $\nu_3(6m(m+1))=2$, \eqref{eq:norm-valuation} gives $\nu_3(t)\geq-1$ when $t\ne0$. Hence, $3t\in\Z$, including when $t=0$, and its fractional part is $0$, $1/3$, or $2/3$.

If $t\notin\Z$, we can write $t=u/3$ with $u\in\Z$ and $3\nmid u$. Substitution in \eqref{eq:J} gives
\[
\frac{m(m+1)u(u-d)}{36}\in\Z.
\]
Since $\nu_3(m(m+1))=1$ and $3\nmid u$, this forces $3\mid(u-d)$. Since $d=2m+1\equiv1\pmod3$, we obtain that every nonintegral $t$ belongs to $1/3+\Z$.

We next show that a class with integral $t$ is $\Gamma$. Suppose instead that $\sqrt2\bm\alpha+\Gamma\ne\Gamma$ and $t=\langle\bm\alpha,\bm\alpha\rangle\in\Z$. Lemma \ref{lem:norm-denominators} (ii) supplies $\sqrt2\bm\beta+\Gamma\in\operatorname{Syl}_{3}(\sqrt{2}\Lambda^*/\Gamma)$ such that $r=2\langle\bm\alpha,\bm\beta\rangle\notin\Z$. Since the squared norms of $\bm\alpha$, $\bm\beta$, and $\bm\alpha+\bm\beta$ all have reduced denominators at most $3$, their difference $r$ belongs to $\frac13\Z$. For $k=0,1,2$, direct expansion gives
\[
\langle\bm\beta+k\bm\alpha,
\bm\beta+k\bm\alpha\rangle
=\langle\bm\beta,\bm\beta\rangle+k^2t+kr.
\]
Because $t\in\Z$ and $r$ is a nonintegral multiple of $1/3$, these three squared norms have fractional parts $0$, $1/3$, and $2/3$ in some order. This contradicts the exclusion of $2/3+\Z$ above. Thus every nonidentity class $\sqrt2\bm\alpha+\Gamma$ in this subgroup satisfies $\langle\bm\alpha,\bm\alpha\rangle\in\frac13+\Z$.

(ii) Suppose that $\operatorname{Syl}_{3}(\sqrt{2}\Lambda^*/\Gamma)\ne\{\Gamma\}$. If $|\operatorname{Syl}_{3}(\sqrt{2}\Lambda^*/\Gamma)|>3$, choose nonidentity classes $\sqrt2\bm\alpha+\Gamma$ and $\sqrt2\bm\beta+\Gamma$ with the latter outside $\{\Gamma,\sqrt2\bm\alpha+\Gamma, -\sqrt2\bm\alpha+\Gamma\}$. Hence both $\sqrt2(\bm\alpha+\bm\beta)+\Gamma$ and $\sqrt2(\bm\alpha-\bm\beta)+\Gamma$ are nonidentity. However,
\[
\langle\bm\alpha+\bm\beta,\bm\alpha+\bm\beta\rangle
+\langle\bm\alpha-\bm\beta,\bm\alpha-\bm\beta\rangle
=2\langle\bm\alpha,\bm\alpha\rangle
+2\langle\bm\beta,\bm\beta\rangle.
\]
By part~(i), the left side belongs to $2/3+\Z$ and the right side to $4/3+\Z$, a contradiction. Hence, $|\operatorname{Syl}_{3}(\sqrt{2}\Lambda^*/\Gamma)|\leq3$. Its order is a power of $3$ and the subgroup is nontrivial, so its order is exactly $3$. Finally, Lemma \ref{lem:even} (iii) gives $\det\Lambda=2^{n-1}|\sqrt{2}\Lambda^*/\Gamma|$, hence $\nu_3(\det\Lambda)=\nu_3(|\sqrt{2}\Lambda^*/\Gamma|)=1$.
\end{proof}

\begin{lemma}\label{lem:symmetric-square-residue}
Suppose that $\nu_3(m(m+1))=1$. Suppose that $\bm v_1,\ldots,\bm v_n\in\Lambda$ form an orthogonal basis of $\mathbb{R}^n$ with squared norms
\[
3u_1,u_2,\ldots,u_n,\qquad u_j\in\Z,\quad 3\nmid u_j.
\]
Suppose that there is a positive integer $h$ with $3\nmid h$ such that $h\cdot \bm x$ is an integer linear combination of $\bm v_1,\ldots,\bm v_n$ for each $\bm x\in X$. Then
\begin{equation}\label{eq:symmetric-square-residue}
\prod_{j=2}^n u_j\equiv-1\pmod3.
\end{equation}
\end{lemma}
\begin{proof}
Denote $N=|X|=\frac{n(n+1)}2$ and write $X=\{\bm x_1,\ldots,\bm x_N\}$. Define the symmetric matrices
\[
\bm F_j=\bm v_j\bm v_j^{\mathsf T}\in\mathbb{R}^{n\times n},\qquad
\bm F_{i,j}=\bm v_i\bm v_j^{\mathsf T}
+\bm v_j\bm v_i^{\mathsf T}\in\mathbb{R}^{n\times n}\quad (i<j).
\]
A simple calculation shows
\[
\langle\bm F_j,\bm F_j\rangle_{\mathrm F}
=\langle\bm v_j,\bm v_j\rangle^2,\qquad
\langle\bm F_{i,j},\bm F_{i,j}\rangle_{\mathrm F}
=2\langle\bm v_i,\bm v_i\rangle\langle\bm v_j,\bm v_j\rangle.
\]
By orthogonality of the $\bm v_j$, the Gram matrix of
\begin{equation}\label{xueq3}
(\bm F_j)_{2\leq j\leq n},\quad
(\bm F_{i,j})_{2\leq i<j\leq n},\quad
\bm F_1,\quad(\bm F_{1,j})_{2\leq j\leq n},
\end{equation}
is the integer diagonal matrix
\begin{equation}\label{xueq4}
\bm D=\operatorname{diag}\bigl(
\bm D_0,\
9u_1^2,\ (6u_1u_j)_{j=2}^n\bigr)\in\mathbb{R}^{N\times N},
\end{equation}
where
\begin{equation}\label{eq:reduced-unit-Gram}
\bm D_0=\operatorname{diag}
\bigl((u_j^2)_{j=2}^n,(2u_i u_j)_{2\leq i<j\leq n}\bigr)\in\mathbb{R}^{(N-n)\times (N-n)}.
\end{equation}
A direct calculation shows that
\[
\det\bm D_0
=2^{\frac{(n-1)(n-2)}{2}}\left(\prod_{j=2}^n u_j\right)^n
\overset{(a)}\equiv-\prod_{j=2}^n u_j\pmod3,
\]
where ($a$) follows from the fact that both $n$ and $\frac{(n-1)(n-2)}{2}$ are odd, and $u_j^2\equiv1\pmod3$. Hence, to prove the lemma, it is enough to prove
\begin{equation}\label{eq:reduced-unit-determinant}
\det\bm D_0\equiv1\pmod3.
\end{equation}

We now prove \eqref{eq:reduced-unit-determinant}. For each $1\leq j\leq N$, write $h\cdot \bm x_j=\sum_{i=1}^n c_{i,j}\bm v_i$ with $c_{i,j}\in\Z$. Then
\[
h^2\bm x_j\bm x_j^{\mathsf T}
=\sum_{i=1}^n c_{i,j}^2\bm F_i
+\sum_{1\leq k<\ell\leq n}c_{k,j}c_{\ell,j}\bm F_{k,\ell}
\in\mathbb{R}^{n\times n},\quad \forall 1\leq j\leq N.
\]
Let $\bm C\in\mathbb{R}^{N\times N}$ be the integer matrix whose $j$-th column lists these coefficients in the displayed order in \eqref{xueq3}. Then $\bm C^{\mathsf T}\bm D\bm C\in\mathbb{R}^{N\times N}$ is the Gram matrix of $h^2\bm x_1\bm x_1^{\mathsf T},\ldots, h^2\bm x_N\bm x_N^{\mathsf T}$. On the other hand, by \eqref{eq:scaled} we have
\begin{equation*}
\langle \bm x_i\bm x_i^{\mathsf T},\bm x_j\bm x_j^{\mathsf T}\rangle_{\mathrm F}=	\langle \bm x_i,\bm x_j \rangle^2
=\begin{cases}
n+2 & \text{if $i=j$,}\\
1	& \text{if $i\neq j$.}
\end{cases}
\end{equation*}
Hence, we have
\begin{equation}\label{eq:symmetric-Gram-identity}
\frac{1}{h^4}\cdot \bm C^{\mathsf T}\bm D\bm C=(n+1)\bm I_N+\bm1\bm1^{\mathsf T},
\end{equation}
where $\bm1\in\mathbb{R}^{N}$ is the column of $N$ ones. Let $\bm R\in\mathbb{R}^{(N-n)\times N}$ be the first $N-n$ rows of $\bm C\in\mathbb{R}^{N\times N}$. We claim that
\begin{equation}\label{xueq7}
\bm R^{\mathsf T}\bm D_0\bm R
\equiv\bm1\bm1^{\mathsf T}\pmod3.
\end{equation}
and
\begin{equation}\label{xueq6}
\operatorname{rank}_{\F_3}\bm R
\geq s:=\frac{N-n+1}{2}
\end{equation}
Then we can choose $s$ columns $\bm z_1,\ldots,\bm z_{s}$ of $\bm R$ that are linearly independent modulo $3$. Equation~\eqref{xueq7} gives
\[
\bm z_i^{\mathsf T}\bm D_0\bm z_j\equiv1\pmod3
\qquad \forall 1\leq i,j\leq s.
\]
The matrix $\bm D_0$ has order $2s-1$ and is nonsingular over $\F_3$. Applying Lemma \ref{lem:constant-pairing-determinant} over the field $\F_3$, we obtain that there is a nonzero $a\in\F_3\setminus\{0\}$ such that
\begin{equation}\label{eq:reduced-unit-determinant2}
\det\bm D_0\equiv (-1)^{s-1} a^2\pmod3.
\end{equation}
Since $m(m+1)$ is even, we have $n=4m(m+1)-1\equiv7\pmod8$ and $N-n=\frac{n(n-1)}{2}\equiv1\pmod4$. Hence, $s$ is odd and $(-1)^{s-1}=1$. Since every nonzero square in $\F_3$ is $1$, we have $a^2\equiv1\pmod3$. Combining \eqref{eq:reduced-unit-determinant2} with $(-1)^{s-1}=1$ and $a^2\equiv1\pmod3$, we arrive at  \eqref{eq:reduced-unit-determinant}.

It remains to prove \eqref{xueq7} and \eqref{xueq6}. We first prove \eqref{xueq7}. All entries of $\bm D_0$ on the diagonal are nonzero modulo $3$, and the last $n$ diagonal entries of $\bm D$ are divisible by $3$. Consequently, for every $i,j$,
\begin{equation}\label{eqxu5}
(\bm C^{\mathsf T}\bm D\bm C)_{i,j}
\equiv\sum_{\ell=1}^{N-n} C_{\ell,i} (\bm D_0)_{\ell,\ell}C_{\ell,j}
=(\bm R^{\mathsf T}\bm D_0\bm R)_{i,j}\pmod3.
\end{equation}
As $h^4\equiv1$ and $n+1=4m(m+1)\equiv0\pmod3$, combining \eqref{eq:symmetric-Gram-identity} and \eqref{eqxu5} we obtain
\begin{equation*}
\bm R^{\mathsf T}\bm D_0\bm R
\equiv \bm C^{\mathsf T}\bm D\bm C
\equiv h^4\bm1\bm1^{\mathsf T}
\equiv \bm1\bm1^{\mathsf T}\pmod3.
\end{equation*}
We arrive at \eqref{xueq7}.

We next prove \eqref{xueq6}. Equation \eqref{eq:symmetric-Gram-identity} implies that $\bm C$ is invertible. Deleting $n$ rows lowers the row-space dimension by at most $n$. Hence,
\begin{equation}\label{eq:reduced-rank-bound}
\operatorname{rank}_{\F_3}\bm R
\geq\operatorname{rank}_{\F_3}\bm C-n
\geq N-\nu_3(\det\bm C)-n,
\end{equation}
where the last inequality follows from Lemma \ref{lem:rank-reduction-three}. We next prove $\nu_3(\det\bm C)=(N-n-1)/2$. Taking determinants in \eqref{eq:symmetric-Gram-identity} gives
\begin{equation}\label{xueq8}
(\det\bm C)^2\det\bm D=h^{4N} (n+1)^{N-1} (n+1+N).
\end{equation}
Recall that $u_j\in\Z$ and $3\nmid u_j$. By the definition of $\bm D$ in \eqref{xueq4}, we have $\nu_3(\det\bm D)=2+(n-1)=n+1$. Since $3\nmid h$, comparing the powers of $3$ in  \eqref{xueq8} we obtain
\begin{equation}\label{xueq9}
2\cdot \nu_3(\det\bm C)+ n+1=(N-1)\cdot \nu_3(n+1)	+ \nu_3(n+1+N).
\end{equation}
Since $n+1=4m(m+1)$, we have $\nu_3(n+1)=1$. Since $n+1+N=(n+1)\frac{d^2}{2}$ and $d^2=n+2\equiv1\pmod3$, we have $\nu_3(n+1+N)=1$. Substituting $\nu_3(n+1)=\nu_3(n+1+N)=1$ into \eqref{xueq9}, we obtain $\nu_3(\det\bm C)=s-1$. Combining with \eqref{eq:reduced-rank-bound}, we arrive at \eqref{xueq6}. This completes the proof.

\end{proof}

\begin{lemma}\label{lem:determinant-residue}
Suppose that $m\equiv0\pmod3$, $\nu_3(m(m+1))=1$, and $\operatorname{Syl}_{3}(\sqrt{2}\Lambda^*/\Gamma)\ne\{\Gamma\}$. Then
\begin{equation}\label{eq:determinant-residue}
\frac{\det\Lambda}{3}\equiv2\pmod3.
\end{equation}
\end{lemma}
\begin{proof}
Choose a lattice basis $\bm e_1,\ldots,\bm e_n$ of $\Lambda$ with the Gram matrix $\bm B=(\langle \bm e_i,\bm e_j\rangle)\in\mathbb{R}^{n\times n}$. Since $\Lambda$ is integral, the Gram matrix $\bm B$ is a nonsingular symmetric integral matrix. By Lemma \ref{lem:rational-diagonalization}, there is an integral matrix $\bm P\in\mathbb{Z}^{n\times n}$ with $3\nmid\det\bm P$ such that
\begin{equation*}
\bm P^{\mathsf T}\bm B\bm P
=\operatorname{diag} (3^{a_1}u_1,\ldots,3^{a_n}u_n)\in\mathbb{R}^{n\times n},
\qquad a_j\in\Z_{\geq0},\quad u_j\in\Z,\quad 3\nmid u_j.
\end{equation*}
Permuting the columns of $\bm P$ permutes the diagonal entries of $\bm P^{\mathsf T}\bm B\bm P$ in the same way and changes $\det\bm P$ at most by a sign, preserving $3\nmid\det\bm P$. Thus, after such a permutation and a corresponding relabeling of the pairs $(a_j,u_j)$, we may assume $a_1\geq\cdots\geq a_n\geq0$. Set $h=|\det\bm P|$ and define
\[
[\bm v_1,\ldots,\bm v_n]=[\bm e_1,\ldots,\bm e_n]\cdot \bm P.
\]
Then $\bm v_1 ,\ldots, \bm v_n$ are orthogonal vectors in $\Lambda$ with $\langle \bm v_j,\bm v_j \rangle=3^{a_j}u_j$. Taking determinants of $\bm P^{\mathsf T}\bm B\bm P$ gives
\begin{equation}\label{xueq1}
\det\Lambda=\det\bm B
=\frac{3^{\sum_{j=1}^n a_j}\prod_{j=1}^n u_j}{h^2}.
\end{equation}
Lemma \ref{lem:three-primary} (ii) gives $\nu_3(\det\Lambda)=1$. Since $3\nmid h$ and $3\nmid u_j$, comparing the powers of $3$ in \eqref{xueq1} yields $\sum_{j=1}^n a_j=1$. Since the $a_j$ are nonnegative integers and $a_1\geq\cdots\geq a_n$, it follows that $a_1=1$ and $a_j=0$ for every $j>1$. Then \eqref{xueq1} implies
\begin{equation}
\frac{\det\Lambda}{3}=\frac{\prod_{j=1}^n u_j}{h^2}.
\end{equation}
Since $\det\Lambda$ is a positive integer and $\nu_3(\det\Lambda)=1$,  we see that $\frac{\det\Lambda}{3}$ is an integer. We claim that
\begin{subequations}
\begin{align}
u_1&\equiv1\pmod3 \label{xueq:u1}\\
\prod_{j=2}^n u_j&\equiv-1\pmod3 \label{xueq:u2}.
\end{align}
\end{subequations}
Since $3\nmid h$, we have $h^2\equiv 1\pmod3$. Then it follows from \eqref{xueq:u1} and \eqref{xueq:u2} that
\[
\frac{\det\Lambda}{3}\equiv-1\equiv2\pmod3,
\]
as desired.

It remains to prove \eqref{xueq:u1} and \eqref{xueq:u2}. We first prove \eqref{xueq:u1}. Denote
\[
\bm\alpha=\frac{h}{3u_1}\cdot \bm v_1.
\]
Using $\bm e_k=\sum_{j=1}^n(\bm P^{-1})_{j,k}\bm v_j$ and orthogonality, we obtain
\[
\langle\bm\alpha,\bm e_k\rangle
=\frac{h}{3u_1}\sum_{j=1}^n(\bm P^{-1})_{j,k}
\langle\bm v_1,\bm v_j\rangle
=h\cdot (\bm P^{-1})_{1,k}\in\Z.
\]
Here, we use the fact that the matrix $h\cdot \bm P^{-1}=\pm\operatorname{adj} (\bm P)$ is integral, since $\bm P$ is integral. Hence $\bm\alpha\in\Lambda^*$, and $\langle\bm\alpha,\bm\alpha\rangle=h^2/(3u_1)$. Substitution in \eqref{eq:J} gives
\[
\frac{m(m+1)h^2(h^2-du_1)}{36u_1^2}\in\Z.
\]
The denominator contains exactly two factors of $3$, while $m(m+1)$ contains exactly one and $3\nmid hu_1$. Therefore $3\mid(h^2-du_1)$. As $h^2\equiv1$ and $d=2m+1\equiv1\pmod3$, this proves $u_1\equiv1\pmod3$.

We next prove \eqref{xueq:u2}. Recall that $3\nmid h$ and the matrix $h\cdot \bm P^{-1}=\pm\operatorname{adj} (\bm P)$ is integral. If $\bm x=\sum_{k=1}^n z_k\bm e_k\in X$ with $z_k\in\Z$, then
\[
h\cdot\bm x
=\sum_{k=1}^n  (h\cdot z_k)\cdot\bm e_k
=\sum_{j=1}^n\left(\sum_{k=1}^n
h\cdot(\bm P^{-1})_{j,k}\cdot z_k\right)\bm v_j.
\]
Every coefficient is an integer. Hence, the same $h$ verifies the hypothesis of Lemma \ref{lem:symmetric-square-residue}, giving
\[
\prod_{j=2}^n u_j\equiv-1\pmod3.
\]
This completes the proof.

\end{proof}

\begin{proof}[Proof of Theorem~\ref{thm:second}]
Suppose that the asserted design exists. By Lemma \ref{lem:local} (ii), every odd-primary subgroup of $\sqrt{2}\Lambda^*/\Gamma$ is trivial except possibly the one at $3$ and those at primes $p\equiv1\pmod{12}$. The latter primes are $1$ modulo $4$, so their Gauss sums are real by Lemma~\ref{lem:real-gauss}. By Lemma \ref{lem:local} (i), every representative in the $2$-primary subgroup has integer squared norm. Hence, each summand in its Gauss sum is $1$ or $-1$, and $G_\Gamma(\operatorname{Syl}_{2}(\sqrt{2}\Lambda^*/\Gamma))\in\R$. If the $3$-primary subgroup were trivial as well, then every primary Gauss sum would be real. Equation \eqref{eq:gauss-primary-product}  gives
\[
G_\Gamma(\sqrt{2}\Lambda^*/\Gamma)
=\prod_{p\text{ prime}}G_\Gamma(\operatorname{Syl}_p(\sqrt{2}\Lambda^*/\Gamma))\in\R,
\]
contrary to \eqref{eq:required-phase}. Hence $\operatorname{Syl}_{3}(\sqrt{2}\Lambda^*/\Gamma)\ne\{\Gamma\}$, and Lemma \ref{lem:three-primary} (ii) and Lemma~\ref{lem:determinant-residue} give
\[
\nu_3(\det\Lambda)=1
\quad\text{and}\quad
\frac{\det\Lambda}{3}\equiv2\pmod3.
\]

On the other hand, Lemma \ref{lem:even} (iii) gives
\[
\det\Lambda=2^{n-1}|\sqrt{2}\Lambda^*/\Gamma|.
\]
For every prime $p\geq5$ with $p\not\equiv1\pmod{12}$, the hypothesis $\nu_p(m(m+1))\leq1$ and Lemma \ref{lem:local} (ii) give $\operatorname{Syl}_p(\sqrt{2}\Lambda^*/\Gamma)=\{\Gamma\}$. The order of this Sylow subgroup is $p^{\nu_p(|\sqrt{2}\Lambda^*/\Gamma|)}$, so for every prime $p\geq5$ with $p\not\equiv1\pmod{12}$, we have
\[
\nu_p(\det\Lambda)
=\nu_p(|\sqrt{2}\Lambda^*/\Gamma|)=0.
\]
Thus, every prime divisor of $\det\Lambda$ other than $2$ and $3$ is $1$ modulo $12$. Since $\det\Lambda$ is a positive integer and $\nu_3(\det\Lambda)=1$, its prime factorization is
\[
\det\Lambda
=2^{\nu_2(\det\Lambda)}\cdot3
\prod_{\substack{p\text{ prime}\\p\mid\det\Lambda\\p\equiv1\ (\mathrm{mod}\ 12)}}
p^{\nu_p(\det\Lambda)}.
\]
Dividing by $3$ removes exactly this single factor of $3$. Moreover, Lemma \ref{lem:local} (i) implies $\nu_2(\det\Lambda)\in2\Z_{\geq0}$, and hence
\[
2^{\nu_2(\det\Lambda)}
=(2^2)^{\nu_2(\det\Lambda)/2}\equiv1\pmod3.
\]
For each prime occurring in the product, $p\equiv1\pmod{12}$ implies $p\equiv1\pmod3$, so $p^{\nu_p(\det\Lambda)}\equiv1\pmod3$. Consequently,
\[
\frac{\det\Lambda}{3}
=2^{\nu_2(\det\Lambda)}
\prod_{\substack{p\text{ prime}\\p\mid\det\Lambda\\p\equiv1\ (\mathrm{mod}\ 12)}}
p^{\nu_p(\det\Lambda)}
\equiv1\pmod3.
\]
The product is finite, and an empty product is understood to be $1$. This contradicts $\frac{\det\Lambda}{3}\equiv2\pmod3$ obtained above, completing the proof.
\end{proof}

\appendix
\section{Proofs of lemmas in Section~\ref{sec:second}}

\subsection{Proof of Lemma \ref{lem:constant-pairing-determinant}}
\label{app:constant-pairing-determinant}
\begin{proof}[Proof of Lemma \ref{lem:constant-pairing-determinant}]

Write $\langle\bm x,\bm y\rangle_{\bm D} :=\bm x^{\mathsf T}\bm D\bm y$ and denote $\bm\ell_i=\bm z_{i+1}-\bm z_1$ for $1\leq i\leq s-1$. Subtracting the first vector from the others is reversible, so $\bm z_1,\bm\ell_1,\ldots,\bm\ell_{s-1}$ are independent. Equation~\eqref{eq:constant-pairing-hypothesis} gives
\[
\langle\bm z_1,\bm z_1\rangle_{\bm D}=1,\qquad
\langle\bm z_1,\bm\ell_i\rangle_{\bm D}=1-1=0,\qquad
\langle\bm\ell_i,\bm\ell_j\rangle_{\bm D}=1-1-1+1=0.
\]
Extend these $s$ vectors to a basis by adjoining $\bm h_1,\ldots,\bm h_{s-1}$. For each $j$, replace $\bm h_j$ by $\bm h_j-\langle\bm z_1,\bm h_j\rangle_{\bm D}\bm z_1$. These elementary column operations preserve the basis. The pairing of the new vector with $\bm z_1$ is
\[
\langle\bm z_1,\bm h_j\rangle_{\bm D}
-\langle\bm z_1,\bm h_j\rangle_{\bm D}
\langle\bm z_1,\bm z_1\rangle_{\bm D}=0.
\]
Continue to denote the new vectors by $\bm h_j$, and set
\[
\bm S=[\,\bm z_1\ \bm\ell_1\ \cdots\ \bm\ell_{s-1}\
\bm h_1\ \cdots\ \bm h_{s-1}\,]\in\mathbb{F}^{(2s-1)\times (2s-1)}.
\]
For any two columns $\bm v,\bm w$ of $\bm S$, the corresponding entry of $\bm S^{\mathsf T}\bm D\bm S$ is $\langle\bm v,\bm w\rangle_{\bm D}$. Hence, the pairings above give
\begin{equation}\label{eq:constant-pairing-block}
 \bm S^{\mathsf T}\bm D\bm S
 =\begin{pmatrix}
    1&\bm0&\bm0\\
    \bm0&\bm0&\bm A\\
    \bm0&\bm A^{\mathsf T}&\bm B
   \end{pmatrix},
\end{equation}
where
\[
\bm A=(a_{i,j})_{i,j=1}^{s-1},\quad
a_{i,j}=\langle\bm\ell_i,\bm h_j\rangle_{\bm D},\qquad
\bm B=(b_{i,j})_{i,j=1}^{s-1},\quad
b_{i,j}=\langle\bm h_i,\bm h_j\rangle_{\bm D}.
\]
The matrix $\bm B$ is symmetric because $\bm D$ is symmetric. Expand the determinant in \eqref{eq:constant-pairing-block} along its first row. In the remaining matrix, exchange columns $i$ and $s-1+i$ for each $i=1,\ldots,s-1$. These $s-1$ column exchanges give
\begin{align*}
(\det\bm S)^2\det\bm D
&=\det\begin{pmatrix}\bm0&\bm A\\
\bm A^{\mathsf T}&\bm B\end{pmatrix}
=(-1)^{s-1}\det\begin{pmatrix}\bm A&\bm0\\
\bm B&\bm A^{\mathsf T}\end{pmatrix}
=(-1)^{s-1} (\det\bm A)^2.
\end{align*}
The left side is nonzero, so $\det\bm A\ne0$. Taking $a=\det\bm A/\det\bm S$ proves the lemma.

\end{proof}

\subsection{Proof of Lemma \ref{lem:rational-diagonalization}}
\label{app:rational-diagonalization}
\begin{proof}[Proof of Lemma \ref{lem:rational-diagonalization}]

We use induction on $n$. If $n=1$, write $b_{1,1}=3^{a_1}u_1$ with $a_1\geq0$ and $3\nmid u_1$, and take $\bm P=(1)$. For $n>1$, write $\bm B=(b_{i,j})_{i,j=1}^n$, where each $b_{i,j}\in \Z$. Let $e\geq0$ be the largest integer such that $3^e$ divides every entry of $\bm B$, i.e.,
\[
e=\min_{\substack{1\leq i,j\leq n\\b_{i,j}\ne0}}\nu_3(b_{i,j}).
\]

We claim that there is an integer matrix $\bm P_0\in\Z^{n\times n}$ such that $\det\bm P_0\in\{1,-1\}$ and the $(1,1)$-entry of $\bm C:=\bm P_0^{\mathsf T}\bm B\bm P_0$ is not divisible by $3^{e+1}$. Note that $3^e$ divides every entry of $\bm C$ because $\bm P_0$ is an integer matrix. Denote $\bm C=(c_{i,j})_{i,j=1}^n$. Write $c_{1,1}=3^{e}u$ and $c_{1,j}=3^e\cdot t_j$ for $2\leq j\leq n$, where $u,t_j\in\Z$ and $3\nmid u$. Define
\[
\bm H=
\begin{pmatrix}
1& \bm t\\
\bm 0&u\cdot \bm I_{n-1}\\
\end{pmatrix}\in\Z^{n\times n},
\]
where $\bm t=(-t_2, -t_3 ,\ldots,-t_n)$. Then we have $\det\bm H=u^{n-1}$ and
\begin{equation*}
\bm H^{\mathsf T}\bm C\bm H
=\begin{pmatrix}3^e u&\bm0\\\bm0&\bm D\end{pmatrix},
\end{equation*}
where $\bm D=(d_{i,j})_{i,j=2}^{n}\in\Z^{(n-1)\times (n-1)}$ and $d_{i,j}:=u^2c_{i,j}-3^e u t_i t_j$. The formula for $d_{i,j}$ shows that $\bm D$ is integral and symmetric. Moreover,
\[
3^e u\det\bm D
=\det(\bm H^{\mathsf T}\bm C\bm H)
 =(\det\bm H)^2(\det\bm P_0)^2\det\bm B\ne0,
\]
so $\bm D$ is nonsingular. Applying the induction hypothesis to $\bm D$, we see that there is an integral matrix $\bm P_1\in\mathbb{Z}^{(n-1)\times (n-1)}$ with $3\nmid\det\bm P_1$ such that
\begin{equation*}
\bm P_1^{\mathsf T}\bm D\bm P_1
=\operatorname{diag} (3^{a_2}u_2,\ldots,3^{a_n}u_n),
\qquad a_j\in\Z_{\geq 0},\quad u_j\in\Z,\quad 3\nmid u_j.
\end{equation*}
Let
\begin{equation*}
\bm P=	\bm P_0\bm H
\begin{pmatrix}
1 & \bm 0 \\
\bm 0 & \bm P_1
\end{pmatrix}\in\mathbb{Z}^{n\times n}.
\end{equation*}
Then
\begin{equation*}
\begin{aligned}
\bm P^{\mathsf T}\bm B\bm P
&=\begin{pmatrix}
1 & \bm 0 \\
\bm 0 & \bm P_1^{\mathsf T}
\end{pmatrix}\bm H^{\mathsf T}\bm P_0^{\mathsf T}\bm B\bm P_0\bm H
\begin{pmatrix}
1 & \bm 0 \\
\bm 0 & \bm P_1
\end{pmatrix}
= \begin{pmatrix}
1 & \bm 0 \\
\bm 0 & \bm P_1^{\mathsf T}
\end{pmatrix}\begin{pmatrix}3^e u&\bm0\\\bm0&\bm D\end{pmatrix}
\begin{pmatrix}
1 & \bm 0 \\
\bm 0 & \bm P_1
\end{pmatrix}\\
&=\begin{pmatrix}3^e u&\bm0\\\bm0&\bm P_1^{\mathsf T}\bm D\bm P_1\end{pmatrix}
=\operatorname{diag} (3^{e}u,3^{a_2}u_2,\ldots,3^{a_n}u_n).
\end{aligned}
\end{equation*}
Note that
\begin{equation*}
\det\bm P=\det\bm P_0 \det\bm H \det\bm P_1=u^{n-1}\det\bm P_0  \det\bm P_1.
\end{equation*}
Since $3\nmid u$, $3\nmid \det\bm P_1$ and $\det\bm P_0\in\{1,-1\}$, we have $3\nmid \det\bm P$. We arrive at our conclusion.

It remains to prove that such  $\bm P_0\in\Z^{n\times n}$ exists. If there is a diagonal entry $b_{k,k}$ not divisible by $3^{e+1}$, then we take $\bm P_0$ as the permutation matrix which permutes the first and the $k$-th column. Then $\det \bm P_0=\pm 1$ and $(\bm P_0^{\mathsf T}\bm B\bm P_0)_{1,1}=b_{k,k}$,  which is not divisible by $3^{e+1}$. We next consider the case when $3^{e+1}$ divides every diagonal entry of $\bm B$. By the definition of $e$, there exists $i\ne j$ with $3^{e+1}\nmid b_{i,j}$. Take $\bm Q_1=\bm I_n+\bm E_{j,i}$, where $\bm E_{j,i}$ denotes the $n\times n$ matrix whose $(j,i)$ entry is $1$ and whose other entries are $0$. Then we have $\det\bm Q_1=1$ and
\[
(\bm Q_1^{\mathsf T}\bm B\bm Q_1)_{i,i}=b_{i,i}+2b_{i,j}+b_{j,j}\not\equiv 0\pmod{3^{e+1}}.
\]
Let $\bm Q_2$ be the permutation matrix which permutes the first and the $i$-th column. Take $\bm P_0=\bm Q_1\bm Q_2$. Then $\bm P_0$ is an integral matrix. Moreover, $\det \bm P_0=\pm 1$ and $(\bm P_0^{\mathsf T}\bm B\bm P_0)_{1,1}=(\bm Q_1^{\mathsf T}\bm B\bm Q_1)_{i,i}=b_{i,i}+2b_{i,j}+b_{j,j}$,  which is not divisible by $3^{e+1}$. This completes the proof.

\end{proof}

\subsection{Proof of Lemma \ref{lem:rank-reduction-three}}
\label{app:rank-reduction-three}
\begin{proof}[Proof of Lemma \ref{lem:rank-reduction-three}]

Denote $\rho=\operatorname{rank}_{\F_3}\bm C$. Reorder the columns $\bm c_1,\ldots,\bm c_t$ so that the first $\rho$ form a basis of the column space modulo $3$. This changes the determinant only by a sign. We continue to denote the reordered matrix by $\bm C$. For $j>\rho$, choose integers $\lambda_{i,j}\in\{0,1,2\}$ with
\[
\bm c_j\equiv\sum_{i=1}^{\rho}\lambda_{i,j}\bm c_i\pmod3.
\]
Every entry of the difference is divisible by $3$, so there is an integer column $\bm e_j$ such that
\[
\bm c_j=\sum_{i=1}^{\rho}\lambda_{i,j}\bm c_i+3\bm e_j.
\]
Expand the determinant in its last $t-\rho$ columns using this equality. Any term selecting some $\bm c_i$, $i\leq\rho$, in a later column has two identical columns and is zero. Therefore
\[
\det\bm C
=3^{t-\rho}
\det(\bm c_1,\ldots,\bm c_\rho,\bm e_{\rho+1},\ldots,\bm e_t).
\]
The last determinant is an integer. Hence, $3^{t-\rho}\mid\det\bm C$, proving the claim. If $\rho=0$, the sums are empty and every column is divisible by $3$. If $\rho=t$, the claim is simply $\nu_3(\det\bm C)\geq0$.

\end{proof}

\section*{Use of Artificial Intelligence}
During the preparation of this manuscript, the author used ChatGPT 5.6 Sol for assistance with filling in technical details and improving the language. All mathematical ideas, proof strategies, and proofs presented in this manuscript are human-generated and were developed and verified by the author. The author takes full responsibility for the correctness and originality of the results.

\end{document}